\documentclass{article}
\usepackage{amsfonts}
\usepackage{amssymb}
\usepackage{amsmath}
\usepackage{amsthm}
\usepackage{cite}
\usepackage{CJK}

\numberwithin{equation}{section}

\newtheorem{theorem}{Theorem}[section]
\newtheorem{lemma}{Lemma}[section]
\newtheorem{remark}{Remark}[section]
\newtheorem{example}{Example}[section]

\date{}

\begin{document}

\title {Meromorphic solutions of a nonlinear differential-difference equation with two exponential terms}

\author{Min Feng CHEN
\thanks{Corresponding author.
This research was supported by the National Natural
Science Foundation of  China (Nos. 12001117, 12501094).}
and Zong Sheng Gao\\
}\maketitle
\date{}

\begin{minipage}{128mm}

\begin{flushleft}
\quad
\end{flushleft}

{\bf Abstract}\quad
Let $n\geq 2$ and $k\geq 1$ be integers. We study finite order meromorphic solutions of
the nonlinear differential-difference equation
\begin{equation*}
f^{n}(z)f'(z)+P_{d}(z,f)=p_{1}(z){\rm e}^{\alpha_{1}z^{k}}+p_{2}(z){\rm e}^{\alpha_{2}z^{k}},
\end{equation*}
where $P_{d}(z,f)\not\equiv 0$ is a difference polynomial in $f$ of degree $d\leq n-1$ whose
coefficients are small meromorphic functions with respect to $f$, $p_{1},~p_{2}$ are nonzero
meromorphic functions with $T(r,p_{j})=o(r^{k})$, $j=1,2$, and $\alpha_{1}\neq\alpha_{2}$ are
nonzero constants. We prove that every such solution has order exactly $k$, and obtain
sufficient conditions on $n,~d,~k$ for the equation to admit no finite order meromorphic
solution with $N(r,\frac{1}{f})=S(r,f)$. Our results extend the known ones from linear exponents $z$
to polynomial exponents $z^{k}$, and replace the linear difference polynomials with the degree
restriction $d\leq n-2$ assumed in the literature by general ones of degree $d\leq n-1$.
Examples show that the restrictions on $n$ and $d$ are sharp.\par

\par

{\bf 2020 MR Subject Classification}\quad 30D30, 39A45, 39B32
\end{minipage}

\section{Introduction and Main results}
\indent\par In this paper, we will use Nevanlinna theory of meromorphic functions (see \cite{CZ1,H,L})
to prove our results. Let $f$ be a meromorphic function in the complex plane.
The notation $S(r,f)$ stands for any quantify satisfying $S(r,f)=o(T(r,f))$,
as $r\to\infty$, possibly outside of a set of $r$ of finite logarithmic measure.
In addition, the notations $\sigma(f)$ and $\lambda\left(\frac{1}{f}\right)$ denote the order of
growth of $f$ and the exponent of convergence of poles of $f$, respectively.
Let $N_{1)}\left(r,\frac{1}{f}\right)$ and $N_{(2}\left(r,\frac{1}{f}\right)$ denote the counting functions
corresponding to simple and multiple zeros of $f$, respectively.
In general, a difference polynomial in $f$ is defined to be a polynomial in $f$ and its shifts $f(z+c)$
($c\in \mathbb{C}\setminus\{0\}$) with small functions as its coefficients, that is, a difference polynomial
$P_{d}(z,f)$ in $f$ is denoted by
\begin{equation}\label{eq1.1}
P_{d}(z,f)=\sum_{\mu\in I}a_{\mu}(z)(f(z))^{\mu_{0}}(f(z+d_{1}))^{\mu_{1}}\cdots(f(z+d_{s}))^{\mu_{s}},
\end{equation}
where $d_{\nu}$ $(\nu=1,2,\ldots,s)$ are distinct non-zero constants, $\mu_{\nu}$ $(\nu=0,1,\ldots,s)$
are non-negative integers, $I$ is a finite set of the index $\mu$, $a_{\mu}(\mu\in I)$
are small meromorphic function of $f$. The degree of $P_{d}(z,f)$ is defined by
$d=\deg P_{d}(z,f)=\max\limits_{\mu\in I}\{|\mu|:|\mu|=\sum\limits_{\nu=0}^{s}|\mu_{\nu}|\}$.\par

\indent The existence of solutions is an important problem in the study of
complex functional equation. Yang and Li\cite{YC} proved that the differential equation
$4f^{3}(z)+3f''(z)=-\sin3z$ has exactly three transcendental entire solutions $f_{1}(z)=\sin z$,
$f_{2}(z)=\frac{\sqrt{3}}{2}\cos z-\frac{1}{2}\sin z$, $f_{2}(z)=-\frac{\sqrt{3}}{2}\cos z-\frac{1}{2}\sin z$.
Comparing a result from Yang and Laine\cite{YL} that the difference equation $f^{3}(z)+q(z)f(z+1)=c\sin bz$
has exactly three distinct entire solution of finite order if $q$ is a nonzero constant, while
if $q$ is a nonconstant polynomial, then it does not admit entire solutions of finite order,
where $b,c$ are non-zero constants. Note that $\sin z$ is a linear combination of the
${\rm e}^{iz}$ and ${\rm e}^{-iz}$,
the research on the more general differential (or difference) equations
\begin{equation}\label{eq1.2}
f^{n}(z)+P_{d}(z,f)=p_{1}{\rm e}^{\alpha_{1}z}+p_{2}{\rm e}^{\alpha_{2}z}
\end{equation}
has attracted much attention (see, e.g.\cite{CZ2,CM1,LY,LP,LYZ,LM2,LU}),
where $P_{d}(z,f)$ is a differential (or difference) polynomial in $f$ of
degree $d$, $p_{1},p_{2}$ are small meromorphic functions of $f$, $\alpha_{1},\alpha_{2}$
are distinct nonzero constants.\par

In 2017, Zhang, Xu and Liao\cite{Z} considered the equation similar to the equation \eqref{eq1.2}
with $f^{n}(z)$ replaced by $f^{n}(z)f'(z)$ and proved the following result.

{\bf Theorem~A}
Let $n\geq 3$ be an integer, $Q_d(z,f)$ be a differential polynomial in $f(z)$ of degree $d\leq n-2$
with rational functions as its coefficients, and let $p_1(z)$, $p_2(z)$ are two nonzero rational
functions, $\alpha_1(z)$, $\alpha_2(z)$ are nonzero polynomials. If the differential equation
\begin{equation}\label{eq1.3}
f^{n}(z)f'(z)+Q_{d}(z,f)=p_{1}(z){\rm e}^{\alpha_{1}(z)}+p_{2}{\rm e}^{\alpha_{2}(z)}
\end{equation}
admits a transcendental meromorphic solution $f(z)$ with finitely many poles, then
$\frac{\alpha'_1}{\alpha'_2}$ is a rational number, and $f(z)$ must be of the form
\begin{equation*}
  f(z)=q(z){\rm e}^{p(z)},
\end{equation*}
where $q(z)$ is a rational function and $p(z)$ is a non-constant polynomial. Furthermore, only
one of the following two cases holds:\\
(1) $\frac{\alpha'_1}{\alpha'_2}=1,Q_d(z,f)\equiv 0$ and $(n+1)p'(z)=\alpha'_1=\alpha'_2$;\\
(2) $\frac{\alpha'_1}{\alpha'_2}=\frac{n+1}{k},k(1\leq k\leq d)$~is an integer,
$Q_d(z,f)\equiv p_2(z){\rm e}^{\alpha_2(z)}$ and $(n+1)p'(z)=\alpha'_1(z)$, or
$\frac{\alpha'_1}{\alpha'_2}=\frac{k}{n+1},~k(1\leq k\leq d)$ is an integer,
$Q_d(z,f)\equiv p_1(z){\rm e}^{\alpha_1(z)}$ and $(n+1)p'(z)=\alpha_{2}'(z)$.\par

In 2019, Liu and Mao \cite{LM1} studied the existence of meromorphic solution of eqaution
\begin{equation}\label{eq1.4}
f^{n}(z)+P_{d}(z,f)=p_{1}(z){\rm e}^{\alpha_{1}(z)}+p_{2}(z){\rm e}^{\alpha_{2}(z)}
\end{equation}
where $\alpha_1(z)$, $\alpha_2(z)$ are nonzero polynomials, and obtained the following result.\par

{\bf Theorem~B}
Let $n(\geq 3)$ and $k(\geq 1)$ be integers, and $P_{d}(z,f)$ be a differential polynomial
in $f$ of degree $d\leq n-2$. Let $\alpha_{1},~\alpha_{2}$ be polynomials of degree $k$,
and $p_{1},~p_{2}$ be non-zero small meromorphic functions of ${\rm e}^{z^{k}}$. If
\begin{equation}\label{eq1.5}
\frac{\alpha_{1}^{(k)}}{\alpha_{2}^{(k)}}\not\in\left\{\frac{n}{n-1},\frac{n-1}{n},-1,1\right\},
\end{equation}
and one of the following conditions occurs:\par
1. $P_{d}(z,0)\not\equiv 0$,\\
\indent 2. $P_{d}(z,0)\equiv 0$,
$\frac{\alpha_{1}^{(k)}}{\alpha_{2}^{(k)}}\not\in\left\{\frac{n}{d},\frac{d}{n}\right\}$,\\
then Eq.\eqref{eq1.4} has no transcendental meromorphic solution of $f$ satisfying $N(r,f)=S(r,f)$.
\par

{\bf Question~1.1}
What can be said about the existence and the growth of finite order meromorphic solutions of
Eq.\eqref{eq1.4} if the dominant term $f^{n}(z)$ is replaced by $f^{n}(z)f'(z)$,
and $P_{d}(z,f)$ is taken to be a difference polynomial in $f$, rather than a differential polynomial?
\par

Motivated by Theorems A and B, we address Question 1.1 in the
case of finite order meromorphic solutions and obtain the following result.
\par

\begin{theorem} \label{th1.1}
Let $n(\geq 2)$ and $k(\geq 1)$ be integers, let $P_{d}(z,f)(\not\equiv 0)$ be a difference polynomial
in $f$ of degree $d\leq n-1$, and let $p_{1}(z),~p_{2}(z)$ be nonzero small meromorphic functions of 
${\rm e}^{z^{k}}$. Suppose $\alpha_{1},~\alpha_{2}$ are distinct nonzero constants, and consider the equation
\begin{equation}\label{eq1.6}
f^{n}(z)f'(z)+P_{d}(z,f)=p_{1}(z){\rm e}^{\alpha_{1}z^{k}}+p_{2}(z){\rm e}^{\alpha_{2}z^{k}}.
\end{equation}
If \eqref{eq1.6} admits a transcendental meromorphic solution $f$ of finite order such
that the coefficients of $P_{d}(z,f)$ are small with respect to $f$, 
then $N(r,f)=S(r,f)$ and one of the following assertions holds:\\
$(1)$ $f(z)=\gamma(z) {\rm e}^{\beta z^{k}}$, where $\gamma(z)$ is a small meromorphic function
with respect to $f(z)$, and $\beta$ is a nonzero constant such that, for some $\{i,j\}=\{1,2\}$,
$(n+1)\beta=\alpha_{i}$, $\gamma^{n}(z)(\gamma'(z)+kz^{k-1}\beta\gamma(z))=p_{i}(z)$.
Moreover, $P_{d}(z,f)\equiv p_{j}(z){\rm e}^{\alpha_{j}z^{k}}$, and
$\frac{\alpha_{i}}{\alpha_{j}}=\frac{n+1}{t}$ for some integer $t$ with $1\leq t\leq d$;\\
$(2)$ $T(r,f)\leq \frac{3}{2}\overline{N}\left(r,\frac{1}{f}\right)+S(r,f)$.
\end{theorem}

\begin{remark}\label{re1.1}
In Theorem \ref{th1.1}, if $P_{d}(z,f)$ is a differential polynomial, then the conclusion
$N(r,f)=S(r,f)$ need not hold in general. For example, the differential equation
$f^{2}(z)f'(z)-13f'(z)-12f''(z)-6f'''(z)={\rm e}^{3z}-25{\rm e}^{z}$
admits a meromorphic solution $f(z)={\rm e}^{z}+\frac{6}{{\rm e}^{z}-1}$ of order one,
for which $N(r,f)=\frac{1}{2}T(r,f)+S(r,f)$.
\end{remark}

The following Examples \ref{ex1.1}-\ref{ex1.3} show that the forms of
meromorphic solutions described in Theorem \ref{th1.1} do occur.\par

\begin{example}\label{ex1.1}
The entire function $f(z)={\rm e}^{z+z^{2}}$ of order two solves the equation
$f^{2}(z)f'(z)+f(z+i\pi)=(1+2z){\rm e}^{3z}\cdot {\rm e}^{3z^{2}}-{\rm e}^{(1+2\pi i)z-\pi^{2}}\cdot {\rm e}^{z^{2}}$.
Here $n=2$, $P_{1}(z,f)\equiv p_{2}(z){\rm e}^{\alpha_{2}z^{2}}$, $\frac{\alpha_{1}}{\alpha_{2}}=\frac{n+1}{t}$
with $t=1=d$. Hence assertion (1) of Theorem \ref{th1.1} is attainable.
\end{example}

\begin{example}\label{ex1.2}
The entire function $f(z)={\rm e}^{z}+{\rm e}^{z^{2}}$ of order two solves the equation
$f^{3}(z)f'(z)-3(1+2z){\rm e}^{2z}f^{2}(z)+\frac{(1+4z){\rm e}^{3z}}{{\rm e}^{2\pi iz-\pi^{2}}+1}
\cdot f(z+i\pi)+\Big[\frac{(1+4z){\rm e}^{3z}}{{\rm e}^{2\pi iz-\pi^{2}}+1}+2(1+3z){\rm e}^{3z}\Big]f(z)
=(1+6z){\rm e}^{z}\cdot {\rm e}^{3z^{2}}+2z{\rm e}^{4z^{2}}$.
Clearly, $T(r,f)=\overline{N}\big(r,\frac{1}{f}\big)+S(r,f)$, which satisfies assertion (2)
of Theorem \ref{th1.1}.
Here $P_{d}(z,0)\equiv 0$, $n=3$, $d=n-1=2$, $k=2$ and $\alpha_{1}=4$, $\alpha_{2}=3$ satisfy
$\frac{\alpha_{1}}{\alpha_{2}}=\frac{4}{3}$.
\end{example}

\begin{example}\label{ex1.3}
The entire function $f(z)={\rm e}^{z^{2}}+{\rm e}^{-z^{2}}$ of order two solves the equation
$f^{3}(z)f'(z)-\frac{8z{\rm e}^{2(2\pi iz-\pi^{2})}}{{\rm e}^{4(2\pi iz-\pi^{2})}-1}f^{2}(z+i\pi)
+\frac{4z(1+{\rm e}^{4(2i\pi z-\pi^{2})})}{{\rm e}^{4(2i\pi z-\pi^{2})}-1}f^{2}(z)
-8z\frac{{\rm e}^{2(2i\pi z-\pi^{2})}-1}{{\rm e}^{2(2i\pi z-\pi^{2})}+1}
=2z{\rm e}^{4z^{2}}-2z{\rm e}^{-4z^{2}}$.
Clearly, $T(r,f)=\overline{N}\big(r,\frac{1}{f}\big)+S(r,f)$, which satisfies assertion (2)
of Theorem \ref{th1.1}.
Here $P_{d}(z,0)\not\equiv 0$, $n=3$, $k=2$, $d=n-1=2$ and $\alpha_{1}=4$, $\alpha_{2}=-4$ satisfy
$\frac{\alpha_{1}}{\alpha_{2}}=-1$.
\end{example}

Examples \ref{ex1.1}--\ref{ex1.3} show that Eq.\eqref{eq1.6} does admit finite
order meromorphic solutions for suitable choices of $\alpha_{1}$ and $\alpha_{2}$.
This naturally leads to the following question: under what
restrictions on $\alpha_{1}$ and $\alpha_{2}$ does Eq.\eqref{eq1.6} admit no
finite order meromorphic solution? We answer this question and obtain the
following two theorems.\par

\begin{theorem} \label{th1.2}
Let $n(\geq 2)$ and $k(\geq 1)$ be integers, and $P_{d}(z,f)$ be a difference polynomial
in $f$ of degree $d\leq n-1$ whose coefficients are meromorphic functions small with
respect to $\mathrm{e}^{z^{k}}$, and such that $P_{d}(z,0)\not\equiv 0$.
Let $p_{1}(z),~p_{2}(z)$ be nonzero small meromorphic functions of ${\rm e}^{z^{k}}$,
and let $\alpha_{1},~\alpha_{2}$ be nonzero constants such that
\begin{equation}\label{eq1.7}
\frac{\alpha_{1}}{\alpha_{2}}\in\Big\{\frac{t}{n+1},\frac{n+1}{t}:1\leq t\leq n-1\Big\}.
\end{equation}
Then Eq.\eqref{eq1.6} admits no meromorphic solution of finite order.
\end{theorem}

\begin{remark}\label{re1.2}
Neither of the two assumptions in Theorem \ref{th1.2} can be omitted:
Example \ref{ex1.1} satisfies \eqref{eq1.7} but $P_{d}(z,0)\equiv 0$, and
Example \ref{ex1.3} satisfies $P_{d}(z,0)\not\equiv 0$ but
$\frac{\alpha_{1}}{\alpha_{2}}\notin\{\frac{t}{n+1},\frac{n+1}{t}:1\leq t\leq n-1\}$;
in both cases Eq.\eqref{eq1.6} does have a finite order meromorphic solution.
\end{remark}

\begin{remark}\label{re1.3}
Theorems \ref{th1.1} and \ref{th1.2} remain valid if Eq.\eqref{eq1.6} is replaced by
\begin{equation}\label{eq1.8}
f^{n}(z)f'(z)+P_{d}(z,f)=p_{1}(z){\rm e}^{\alpha_{1}(z)}+p_{2}(z){\rm e}^{\alpha_{2}(z)},
\end{equation}
where $\alpha_{1}(z)=a_0z^{k}+a_1z^{k-1}+\cdots+a_{k-1}z+a_k$,
$\alpha_{2}(z)=b_0z^{k}+b_1z^{k-1}+\cdots+b_{k-1}z+b_k$, $a_0,~b_0$ are nonzero constants
with $a_0\neq b_0$. Indeed, writing $p_{1}(z){\rm e}^{\alpha_{1}(z)}+p_{2}(z){\rm e}^{\alpha_{2}(z)}
=p_{1}(z){\rm e}^{\alpha_{1}(z)-a_0z^{k}}{\rm e}^{a_0z^{k}}+
p_{2}(z){\rm e}^{\alpha_{2}(z)-b_0z^{k}}{\rm e}^{b_0z^{k}}
:=q_{1}(z){\rm e}^{a_0z^{k}}+q_{2}(z){\rm e}^{b_0z^{k}}$,
we see that $q_{1},q_{2}$ are nonzero meromorphic functions. Since
$T(r,p_j)=o(r^{k})$, $\deg(\alpha_{1}(z)-a_0z^{k})\leq k-1$ and
$\deg(\alpha_{2}(z)-b_0z^{k})\leq k-1$, it follows that $T(r,q_j)=o(r^{k})$,
$j=1,2$.
\end{remark}

\begin{theorem} \label{th1.3}
Let $n(\geq 3)$  be an integer, let $P_{d}(z,f)$ be a difference polynomial
in $f$ of degree $d\leq n-2$ whose coefficients are polynomials, and let 
$\lambda,~p_{j}~(j=1,2)$ be nonzero constants. Then the equation
\begin{equation}\label{eq1.9}
f^{n}(z)f'(z)+P_{d}(z,f)=p_{1}{\rm e}^{\lambda z}+p_{2}{\rm e}^{-\lambda z}
\end{equation}
admits no meromorphic solution of finite order..
\end{theorem}

\begin{remark}\label{re1.4}
The restrictions $n\geq 3$ and $d\leq n-2$ in Theorem \ref{th1.3} are sharp,
as illustrated by Examples \ref{ex1.4} and \ref{ex1.5} below.
\end{remark}

\begin{example}\label{ex1.4}
The entire function $f(z)={\rm e}^{z}+{\rm e}^{-z}$ solves the equation
$f^{2}(z)f'(z)-if(z-i\frac{\pi}{2})={\rm e}^{3z}-{\rm e}^{-3z}$.
Here $n=2$, $d=n-1=1$ and $P_{d}(z,0)\equiv 0$.
One readily sees that $T(r,f)=\overline{N}\big(r,\frac{1}{f}\big)+S(r,f)$.
\end{example}

\begin{example}\label{ex1.5}
The entire function $f(z)={\rm e}^{z}+{\rm e}^{-z}$ solves the equation
$f^{3}(z)f'(z)+2if(z+i\frac{\pi}{2})f(z)={\rm e}^{4z}-{\rm e}^{-4z}$.
Here $n=3$, $d=n-1=2$ and $P_{d}(z,0)\equiv 0$.
One readily sees that $T(r,f)=\overline{N}\big(r,\frac{1}{f}\big)+S(r,f)$.
\end{example}

\section{Some Lemmas}
\begin{lemma}(\cite[Lemma~2.2]{CM2})\label{le2.1}
Let $f$ be a transcendental meromorphic solution of the differential-difference equation
\begin{equation}\label{eq2.1}
f^n(z)f'(z)+P_d(z,f)=H(z),
\end{equation}
where $n(\geq 2)$ is an integer, $P_d(z,f)$ is a difference polynomial in $f$ of degree
$d\leq n-1$, and $H(z)$ is meromorphic function satisfying $N(r,H)=S(r,f)$. If
$\sigma(f)<+\infty$, then $N(r,f)=S(r,f)$, and $\sigma(f)=\sigma(H)$.
\end{lemma}

\begin{lemma}\label{le2.2}
Let $n(\geq 2)$ and $k(\geq 1)$ be integers, $\alpha_{1},\alpha_{2}$
be distinct nonzero constants, and let $p_{1}(z),~p_{2}(z)$ be nonzero meromorphic functions.
Then the equation
\begin{equation}\label{eq2.2}
f^{n}(z)=p_{1}(z){\rm e}^{\alpha_{1}z^{k}}+p_{2}(z){\rm e}^{\alpha_{2}z^{k}}
\end{equation}
admits no meromorphic solution $f(z)$ satisfying $T(r,p_{j})=S(r,f)~(j=1,2)$.
\end{lemma}
\begin{proof} Suppose that \eqref{eq2.2} has a meromorphic solution $f(z)$ such that
$T(r,p_{j})=S(r,f)~(j=1,2)$. Since $p_{1}(z){\rm e}^{\alpha_{1}z^{k}}+p_{2}(z){\rm e}^{\alpha_{2}z^{k}}$
is transcendental, then every meromorphic solution $f(z)$ of \eqref{eq2.2} is also transcendental.
By \eqref{eq2.2}, we have
\begin{equation*}
nN(r,f)=N(r,f^{n})=N(r,p_{1}(z){\rm e}^{\alpha_{1}z^{k}}+p_{2}(z){\rm e}^{\alpha_{2}z^{k}})=S(r,f),
\end{equation*}
which implies that $N(r,f)=S(r,f)$. In addition, it follows from \eqref{eq2.2} that
\begin{equation}\label{eq2.3}
nT(r,f)=T(r,f^{n})=T(r,p_{1}(z){\rm e}^{\alpha_{1}z^{k}}+p_{2}(z){\rm e}^{\alpha_{2}z^{k}})
\leq \frac{|\alpha_{1}|+|\alpha_{2}|}{\pi}r^{k}+S(r,f).
\end{equation}
Differentiating \eqref{eq2.2} and eliminating ${\rm e}^{\alpha_{1}z^{k}}$ yields
\begin{equation}\label{eq2.4}
np_{1}f^{n-1}f'-(p_{1}'+\alpha_{1}kz^{k-1}p_{1})f^{n}=A{\rm e}^{\alpha_{2}z^{k}},
\end{equation}
where $A=p_{1}p_{2}'-p_{1}'p_{2}+(\alpha_{2}-\alpha_{1})kz^{k-1}p_{1}p_{2}$ and $A\not\equiv 0$.
Otherwise, ${\rm e}^{(\alpha_{1}-\alpha_{2})z^{k}}=c_0p_{2}/p_{1}$,
$c_0\in \mathbb{C}\setminus\{0\}$. By \eqref{eq2.3}, we get
$T(r,{\rm e}^{(\alpha_{1}-\alpha_{2})z^{k}})=T(r,c_0p_{2}/p_{1})=S(r,f)\leq S(r,{\rm e}^{z^{k}})$,
which implies that $\alpha_{1}=\alpha_{2}$, a contradiction. Since $N(r,f)=S(r,f)$,
it follows from \eqref{eq2.4} that
\begin{equation}\label{eq2.5}
T(r,{\rm e}^{\alpha_{2}z^{k}})=m(r,{\rm e}^{\alpha_{2}z^{k}})\leq nm(r,f)+S(r,f)
=nT(r,f)+S(r,f).
\end{equation}
Combining \eqref{eq2.3} and \eqref{eq2.5}, we get
\begin{equation}\label{eq2.6}
\sigma(f)=k,\quad S(r,{\rm e}^{z^{k}})=S(r,f).
\end{equation}
Suppose that $z_{0}$ is a zero of $p_{1}(z){\rm e}^{\alpha_{1}z^{k}}+p_{2}(z){\rm e}^{\alpha_{2}z^{k}}$
with multiplicity $m\geq2$, then $z_{0}$ satisfies the system of equations
\begin{equation}\label{eq2.7}
p_{1}(z){\rm e}^{\alpha_{1}z^{k}}+p_{2}(z){\rm e}^{\alpha_{2}z^{k}}=0, \quad
(p_{1}'(z)+\alpha_{1}kz^{k-1}p_{1}(z)){\rm e}^{\alpha_{1}z^{k}}+
(p_{2}'(z)+\alpha_{2}kz^{k-1}p_{2}(z)){\rm e}^{\alpha_{2}z^{k}}=0.
\end{equation}
Then $z_{0}$ must be a zero of
$A(z)=p_{1}(z)p_{2}'(z)-p_{1}(z)'p_{2}(z)+(\alpha_{2}-\alpha_{1})kz^{k-1}p_{1}(z)p_{2}(z)$.
Note that $T(r,A)=S(r,f)$, this leads to
\begin{equation}\label{eq2.8}
N_{(2}\Big(r,\frac{1}{p_{1}(z){\rm e}^{\alpha_{1}z^{k}}+p_{2}(z){\rm e}^{\alpha_{2}z^{k}}}\Big)=S(r,f).
\end{equation}
On the other hand, since $n\geq 2$, by \eqref{eq2.2}, we see that every zero of $f(z)$ is necessarily a
multiple zero of $p_{1}(z){\rm e}^{\alpha_{1}z^{k}}+p_{2}(z){\rm e}^{\alpha_{2}z^{k}}$.
It follows from \eqref{eq2.8} that
\begin{equation}\label{eq2.9}
N\Big(r,\frac{1}{f}\Big)=S(r,f).
\end{equation}
By \eqref{eq2.2}, \eqref{eq2.6}, \eqref{eq2.9} and the second main theorem, we have
\begin{eqnarray*}
T(r,{\rm e}^{(\alpha_{1}-\alpha_{2})z^{k}})&\leq &
N\Big(r,\frac{1}{{\rm e}^{(\alpha_{1}-\alpha_{2})z^{k}}}\Big)
+N(r,{\rm e}^{(\alpha_{1}-\alpha_{2})z^{k}})
+N\Big(r,\frac{1}{{\rm e}^{(\alpha_{1}-\alpha_{2})z^{k}}+\frac{p_{2}}{p_{1}}}\Big)\\
&& +S(r,{\rm e}^{z^{k}})\\
&\leq & N\Big(r,\frac{1}{p_{1}{\rm e}^{\alpha_{1}z^{k}}+p_{2}{\rm e}^{\alpha_{2}z^{k}}}\Big)
+S(r,{\rm e}^{z^{k}})\\
&= & N\Big(r,\frac{1}{f^{n}}\Big)+S(r,{\rm e}^{z^{k}})=S(r,{\rm e}^{z^{k}}),
\end{eqnarray*}
this leads to a contradiction with
$T(r,{\rm e}^{(\alpha_{1}-\alpha_{2})z^{k}})=\frac{|\alpha_{1}-\alpha_{2}|}{\pi}r^{k}$.
Then Eq.\eqref{eq2.2} admits no meromorphic solution $f(z)$ satisfying $T(r,p_{j})=S(r,f)~(j=1,2)$.
\end{proof}

\begin{remark}\label{re2.1}
Lemma \ref{le2.2} remains valid if Eq.\eqref{eq2.1} is replaced by
$f^{n}(z)=p_{1}(z){\rm e}^{\alpha_{1}(z)}+p_{2}(z){\rm e}^{\alpha_{2}(z)}$,
where $\alpha_{1}(z),\alpha_{2}(z)$ are polynomials of degree k and satisfy
$\frac{\alpha_{1}^{(k)}}{\alpha_{2}^{(k)}}\neq 1$.
\end{remark}

\begin{lemma}\label{le2.3}
Let $n(\geq 2)$ and $k(\geq 1)$ be integers, $\alpha_{1},\alpha_{2}$
be distinct nonzero constants, and let $p_{1}(z),~p_{2}(z)$ be nonzero meromorphic functions.
Then the equation
\begin{equation}\label{eq2.10}
f^{n}(z)f'(z)=p_{1}(z){\rm e}^{\alpha_{1}z^{k}}+p_{2}(z){\rm e}^{\alpha_{2}z^{k}}
\end{equation}
admits no meromorphic solution $f(z)$ satisfying $T(r,p_{j})=S(r,f)~(j=1,2)$.
\end{lemma}
\begin{proof} Suppose that \eqref{eq2.10} has a meromorphic solution $f$ such that
$T(r,p_{j})=S(r,f)~(j=1,2)$. Since $p_{1}(z){\rm e}^{\alpha_{1}z^{k}}+p_{2}(z){\rm e}^{\alpha_{2}z^{k}}$
is transcendental, then every meromorphic solution $f(z)$ of \eqref{eq2.10} is also transcendental.
An argument analogous to that of Lemma \ref{le2.2} yields
$N(r,f)=S(r,f)$ and $N(r,\frac{1}{f})=S(r,f)$. By \eqref{eq2.10}, we have
\begin{equation}\label{eq2.11}
N\left(r,\frac{p_{1}(z){\rm e}^{\alpha_{1}z^{k}}+p_{2}(z){\rm e}^{\alpha_{2}z^{k}}}{f^{n+1}}\right)
=N\left(r,\frac{f^{n}f'}{f^{n+1}}\right)=\overline{N}(r,f)+\overline{N}\left(r,\frac{1}{f}\right)=S(r,f),
\end{equation}
\begin{equation}\label{eq2.12}
m\left(r,\frac{p_{1}(z){\rm e}^{\alpha_{1}z^{k}}+p_{2}(z){\rm e}^{\alpha_{2}z^{k}}}{f^{n+1}}\right)
=m\left(r,\frac{f^{n}f'}{f^{n+1}}\right)=S(r,f).
\end{equation}
It follows from \eqref{eq2.11} and \eqref{eq2.12} that
\begin{equation}\label{eq2.13}
T\left(r,\frac{p_{1}(z){\rm e}^{\alpha_{1}z^{k}}+p_{2}(z){\rm e}^{\alpha_{2}z^{k}}}{f^{n+1}}\right)
=S(r,f).
\end{equation}
Set $h(z)=\frac{f^{n+1}(z)}{p_{1}(z){\rm e}^{\alpha_{1}z^{k}}+p_{2}(z){\rm e}^{\alpha_{2}z^{k}}}$,
then
\begin{equation}\label{eq2.14}
f^{n+1}(z)=p_{1}(z)h(z){\rm e}^{\alpha_{1}z^{k}}+p_{2}(z)h(z){\rm e}^{\alpha_{2}z^{k}},
\end{equation}
where $n+1\geq 3$ and $T(r,p_{j}h)=S(r,f)~(j=1,2)$.
By Lemma \ref{le2.2}, the Eq.\eqref{eq2.14} cannot admit a meromorphic solution $f(z)$ such that
$T(r,p_{j}h)=S(r,f)~(j=1,2)$. This contradicts the assumption.
\end{proof}

\begin{remark}\label{re2.2}
Lemma \ref{le2.3} remains valid if Eq.\eqref{eq2.10} is replaced by
$f^{n}(z)f'(z)=p_{1}(z){\rm e}^{\alpha_{1}(z)}+p_{2}(z){\rm e}^{\alpha_{2}(z)}$,
where $\alpha_{1}(z),\alpha_{2}(z)$ are polynomials of degree k and satisfy
$\frac{\alpha_{1}^{(k)}}{\alpha_{2}^{(k)}}\neq 1$.
\end{remark}

\begin{lemma}(\cite[Theorem~1.51]{Y})\label{le2.4}
Suppose that $f_{1}, f_{2},\ldots,f_{n}(n\geq 2)$ are meromorphic functions and
$g_{1}, g_{2},\ldots,g_{n}$ are entire functions satisfying the following conditions:\\
$(1)$ $\sum_{j=1}^{n}\limits f_{j}e^{g_{j}}\equiv 0.$\\
$(2)$ $g_{j}-g_{k}$ are not constants for $1\leq j<k\leq n$.\\
$(3)$ For $1\leq j\leq n, 1\leq h<k\leq n$,
$$T(r,f_{j})=o(T(r,e^{g_{h}-g_{k}}))~(r\rightarrow \infty, r\not\in E),$$
where $E$ is a set with finite linear measure or logarithmic measure.\\
Then $f_{j}\equiv 0~(j=1,\ldots, n)$.
\end{lemma}

\begin{lemma}(\cite[Corollary~3.3]{HK})\label{le2.5}
Let $f$ be a non-constant finite order meromorphic solution of
$f^{n}(z)P(z,f)=Q(z,f)$,
where $P(z,f)$ and $Q(z,f)$ are difference polynomials
in $f$ with small meromorphic coefficients, and let $c\in \mathbb{C}$, $\delta<1$.
If the total degree of $Q(z,f)$ as a polynomial
in $f$ and its shifts is $\leq n$, then
\begin{equation*}
m(r,P(z,f))=O\left(\frac{T(r+|c|,f)}{r^{\delta}}\right)+o(T(r,f))
\end{equation*}
for all $r$ outside of a possible exceptional set with finite logarithmic measure.
\end{lemma}

\begin{remark}\label{re2.3}
Lemma \ref{le2.5} is still holds for the case $P(z,f)$, $Q(z,f)$ being differential-difference
polynomials in $f$ with small meromorphic coefficients. In particular, $\sigma(f)<\infty$,
if the coefficients of $P(z,f)$ and $Q(z,f)$ are polynomials or rational functions, then
\begin{equation*}
m(r,P(z,f))=S(r,f)+O(\log r),
\end{equation*}
where $r$ is sufficiently large.
\end{remark}

\begin{lemma}(\cite[Corollary~2.6]{CF})\label{le2.6}
Let $\eta_{1},\eta_{2}$ be two complex numbers such that $\eta_{1}\neq\eta_{2}$
and let $f$ be a finite order meromorphic function. Let $\sigma$ be the order
of $f$, then for each $\varepsilon>0$, we have
\begin{equation*}
m\left(r,\frac{f(z+\eta_{1})}{f(z+\eta_{2})}\right)=O(r^{\sigma-1+\varepsilon})+O(\log r).
\end{equation*}
\end{lemma}

\begin{lemma}(\cite[Lemma~2.5]{M})\label{le2.7}
Let $m,k$ be positive integers, $\alpha_{1},\alpha_{2},\ldots,\alpha_{m}$ be distinct nonzero complex numbers,
and $A_{0},A_{1},\ldots,A_{m}$ be meromorphic functions of order less than $k$ such that
$A_{j}\not\equiv 0(1\leq j\leq m)$. Set $\varphi(z)=A_{0}(z)+\sum_{j=1}^{m}A_{j}(z){\rm e}^{\alpha_{j}z^{k}}$,
then the following results hold.\par
(i) There exist two positive numbers $\delta_{1}<\delta_{2}$, such that for sufficiently large $r$,
$\delta_{1}r^{k}\leq T(r,\varphi)\leq \delta_{2}r^{k}$.\\
\indent (ii) If $A_{0}\not\equiv 0$, then $m(r,\frac{1}{\varphi})=S(r,{\rm e}^{z^{k}})~(r\to \infty)$.
\end{lemma}

\begin{lemma}(\cite[Lemma~6]{LP})\label{le2.8}
Suppose that $f$ is a transcendental meromorphic function, $a,b,c,d$ are small functions of $f$
with $acd\not\equiv 0$. If
\begin{align*}
af^{2}+bff'+c(f')^{2}=d,
\end{align*}
then
\begin{align*}
c(b^{2}-4ac)\frac{d'}{d}+b(b^{2}-4ac)-c(b^{2}-4ac)'+(b^{2}-4ac)c'=0.
\end{align*}
\end{lemma}
\begin{remark}\label{re2.2}
The condition $acd\not\equiv 0$ in Lemma \ref{le2.8} is not necessary and it can be
replaced by $cd\not\equiv 0$. Please check the proof of Lemma $6$ in \cite{LP}.
\end{remark}

\section{ Proof of Theorem \ref{th1.1}}
\begin{proof}
Let $f$ be a meromorphic solution of \eqref{eq1.6}.
Since $p_1{\rm e}^{\alpha_1z^{k}}+p_2{\rm e}^{\alpha_2z^{k}}$ is transcendental,
then every finite order meromorphic solution $f$ of \eqref{eq1.6} is
also transcendental. Note that
$N(r,p_1{\rm e}^{\alpha_1z^{k}}+p_2{\rm e}^{\alpha_2z^{k}})=S(r,{\rm e}^{z^{k}})$
and $\sigma(p_1{\rm e}^{\alpha_1z^{k}}+p_2{\rm e}^{\alpha_2z^{k}})=k$,
by Lemma \ref{le2.1} and the argument therein, we have
\begin{equation}\label{eq3.1}
S(r,f)=S(r,{\rm e}^{z^{k}}),\quad N(r,f)=S(r,f),\quad \sigma(f)=k.
\end{equation}
\indent Set $F(z)=f^n(z)f'(z)$, $P(z)=P_d(z,f)$. Differentiating \eqref{eq1.6} yields
\begin{equation}\label{eq3.2}
F'+P'=(\alpha_1kz^{k-1}p_1+p'_1){\rm e}^{\alpha_1z^{k}}+(\alpha_2kz^{k-1}p_2+p'_2){\rm e}^{\alpha_2z^{k}}.
\end{equation}
Eliminating ${\rm e}^{\alpha_{1}z^{k}}$ or ${\rm e}^{\alpha_{2}z^{k}}$ from \eqref{eq1.6} and \eqref{eq3.2}
respectively yields
\begin{equation}\label{eq3.3}
p_1(F'+P')-(\alpha_1kz^{k-1}p_1+p'_1)(F+P)=W{\rm e}^{\alpha_2z^{k}},
\end{equation}
and
\begin{equation}\label{eq3.4}
p_2(F'+P')-(\alpha_2kz^{k-1}p_2+p'_2)(F+P)=-We^{\alpha_1z^{k}},
\end{equation}
where
\begin{equation}\label{eq3.5}
W=p_1p'_2-p'_1p_2+(\alpha_2-\alpha_1)kz^{k-1}p_1p_2,
\end{equation}
and $W\not\equiv 0$. Otherwise, $W\equiv 0$, then $p_2/p_1=c_1{\rm e}^{(\alpha_1-\alpha_2)z^{k}}$,
$c_1\in \mathbb{C}\setminus\{0\}$. Note that $T(r,p_j)=S(r,{\rm e}^{z^{k}})~(j=1,2)$,
then $\alpha_1=\alpha_2$, a contradiction.
Differentiating \eqref{eq3.3} and eliminating ${\rm e}^{\alpha_{2}z^{k}}$, we obtain
\begin{equation}\label{eq3.6}
W_1F''+W_2F'+W_3F=-(W_1P''+W_2P'+W_3P),
\end{equation}
where
\begin{equation}\label{eq3.7}
\left\{\begin{array}{l}
W_1=p_1W,\\
W_2=-(\alpha_1+\alpha_2)kz^{-1}p_1W-p_1W',\\
W_3=-[(p'_1+\alpha_1kz^{k-1}p_1)'-\alpha_2kz^{k-1}(p'_1+\alpha_1kz^{k-1}p_1)]W+(p'_1+\alpha_1kz^{k-1}p_1)W',
\end{array}\right.
\end{equation}
$T(r,W_{j})=S(r,f)~(j=1,2,3)$ and $W_3\not\equiv 0$. Otherwise, $W_3\equiv 0$, then
\begin{equation*}
\alpha_2kz^{k-1}=\frac{(p'_1+\alpha_1kz^{k-1}p_1)'}{p'_1+\alpha_1kz^{k-1}p_1}-\frac{W'}{W},
\end{equation*}
which implies that
\begin{equation*}
T(r,{\rm e}^{\alpha_2z^{k}})
=T\Big(r,\frac{p'_1+\alpha_1kz^{k-1}p_1}{W}\Big)=S(r,f),
\end{equation*}
a contradiction.\par
Now we consider the following two cases:\par
{\bf Case 1:}\ \ Suppose that $W_1F''+W_2F'+W_3F\equiv 0$. Substituting (\ref{eq3.7}) into this identity yields
\begin{align}\label{eq3.8}
\begin{array}{l}
p_1WF''-[(\alpha_1+\alpha_2)kz^{k-1}p_1W+p_1W']F'-\\
\{[(p'_1+\alpha_1kz^{k-1}p_1)'-\alpha_2kz^{k-1}(p'_1+\alpha_1kz^{k-1}p_1)]W
-(p'_1+\alpha_1kz^{k-1}p_{1})W'\}F\equiv 0.
\end{array}
\end{align}
Since
\begin{align}\label{eq3.9}
\begin{array}{l}
Wp_1[(p'_1+\alpha_1kz^{k-1}p_1)'+\alpha_1kz^{k-1}(p'_1+\alpha_1kz^{k-1}p_1)]\\
-Wp_1\big[(\alpha_1+\alpha_2)kz^{k-1}+\frac{W'}{W}\big](p'_1+\alpha_1kz^{k-1}p_1)\\
-Wp_1\big\{[(p'_1+\alpha_1kz^{k-1}p_1)'-\alpha_2kz^{k-1}(p'_1+\alpha_1kz^{k-1}p_1)]
-(p'_1+\alpha_1kz^{k-1}p_1)\frac{W'}{W}\big\}\equiv 0.
\end{array}
\end{align}
It follows from \eqref{eq3.5} that
\begin{align}\label{eq3.10}
\begin{array}{ll}
Wp_1[(p'_2+\alpha_2kz^{k-1}p_2)'+\alpha_2kz^{k-1}(p'_2+\alpha_2kz^{k-1}p_2)]\\
-Wp_1\big[(\alpha_1+\alpha_2)kz^{k-1}+\frac{W'}{W}\big](p'_2+\alpha_2kz^{k-1}p_2)\\
-Wp_2\big\{[(p'_1+\alpha_1kz^{k-1}p_1)'-\alpha_2kz^{k-1}(p'_1+\alpha_1kz^{k-1}p_1)]
-(p'_1+\alpha_1kz^{k-1}p_1)\frac{W'}{W}\big\}\\
=W[p_1(p'_2+\alpha_2kz^{k-1}p_2)'+\alpha_2kz^{k-1}p_1(p'_2+\alpha_2kz^{k-1}p_2)]\\
-W\big[(\alpha_1+\alpha_2)kz^{k-1}+\frac{W'}{W}\big]p_1(p'_2+\alpha_2kz^{k-1}p_2)\\
-W\big\{[p_2(p'_1+\alpha_1kz^{k-1}p_1)'-\alpha_2kz^{k-1}p_2(p'_1+\alpha_1kz^{k-1}p_1)]
-p_2(p'_1+\alpha_1kz^{k-1}p_1)\frac{W'}{W}\big\}\\
=W\big\{p_1(p'_2+\alpha_2kz^{k-1}p_2)'-\alpha_1kz^{k-1}p_1p'_2
-p_2(p'_1+\alpha_1kz^{k-1}p_1)'+\alpha_2kz^{k-1}p'_1p_2\\
-[-p'_1p_2+p_1p'_2+(\alpha_2-\alpha_1)kz^{k-1}p_1p_2]\frac{W'}{W}\big\}\\
=W\big(W'-W\cdot\frac{W'}{W}\big)\equiv 0.
\end{array}
\end{align}
It follows from \eqref{eq3.9} and \eqref{eq3.10} that $p_1(z){\rm e}^{\alpha_1z^{k}}$
and $p_2(z){\rm e}^{\alpha_2z^{k}}$ are two linearly independent solutions of \eqref{eq3.8}.
Then the general solution of \eqref{eq3.8} is of the form
\begin{equation}\label{eq3.11}
F(z)=\eta_1p_1(z){\rm e}^{\alpha_1z^{k}}+\eta_2p_2(z){\rm e}^{\alpha_2z^{k}},
\end{equation}
where $\eta_1,\eta_2$ are constants. Since $F(z)=f^{n}(z)f'(z)$, by Lemma \ref{le2.3},
we have
\begin{equation}\label{eq3.12}
f^n(z)f'(z)=\eta_1p_1(z){\rm e}^{\alpha_1z^{k}}\quad\text{or}\quad
f^n(z)f'(z)=\eta_2p_2(z){\rm e}^{\alpha_2z^{k}}.
\end{equation}
Note that $n\geq 2$, it follows from \eqref{eq3.12} that any zero or pole of $f$ must also be
a zero or pole of $p_1$ (or $p_2$), which implies that $N(r,f)+N(r,1/f)\leq N(r,p_j)+N(r,1/p_j)=S(r,f),~j=1,2$.
If $f^n(z)f'(z)=\eta_1p_1(z){\rm e}^{\alpha_1z^{k}}$, then
\begin{equation*}
N\Bigg(r,\frac{\eta_1p_1{\rm e}^{\alpha_1z^{k}}}{f^{n+1}}\Bigg)
=N\Bigg(r,\frac{f^{n}f'}{f^{n+1}}\Bigg)=\overline{N}(r,f)+\overline{N}\Big(r,\frac{1}{f}\Big)=S(r,f),
\end{equation*}
\begin{equation*}
m\Bigg(r,\frac{\eta_1p_1{\rm e}^{\alpha_1z^{k}}}{f^{n+1}}\Bigg)
=m\Bigg(r,\frac{f^{n}f'}{f^{n+1}}\Bigg)=S(r,f),
\end{equation*}
consequently, $T\Big(r,\frac{\eta_1p_1{\rm e}^{\alpha_1z^{k}}}{f^{n+1}}\Big)=S(r,f)$.
Set $g(z)=\frac{f^{n+1}(z)}{\eta_1p_1(z){\rm e}^{\alpha_1z^{k}}}$, then $T(r,g)=S(r,f)$,
whence $f^{n+1}(z)=\eta_1p_1(z)g(z){\rm e}^{\alpha_1z^{k}}$,
$f(z)=\gamma_1(z){\rm e}^{\frac{\alpha_1}{n+1}z^{k}}$. Substituting these expressions into
\eqref{eq1.6}, we get
\begin{equation}\label{eq3.13}
\eta_1p_1(z){\rm e}^{\alpha_1z^{k}}+\sum\limits_{j=0}^d\Psi_j(z){\rm e}^{\frac{j\alpha_1}{n+1}z^{k}}
=p_1(z){\rm e}^{\alpha_1z^{k}}+p_2(z){\rm e}^{\alpha_2z^{k}},
\end{equation}
where $T(r,\Psi_j(z))=S(r,f)(j=0,1,2,\ldots,d)$. By \eqref{eq3.13} and Lemma \ref{le2.4}, we obtain
$\eta_1=1,\Psi_0(z)=P_d(z,0)\equiv 0$. Moreover, $P_d(z,f)\equiv p_2(z){\rm e}^{\alpha_2z^{k}}$
and there is an integer $t\in\{1,2,\ldots,d\}$ such that
$\frac{t\alpha_1}{n+1}=\alpha_2$; equivalently, $\alpha_{1}/\alpha_{2}=(n+1)/t$.
Consequently, $\gamma_1^n\left(\gamma'_1+\frac{\alpha_1}{n+1}kz^{k-1}\gamma_1\right)=p_1$.\par

Similarly, if $f^n(z)f'(z)=\eta_2p_2(z){\rm e}^{\alpha_2z^{k}}$, then we obtain
$P_d(z,0)\equiv 0$ and $P_d(z,f)\equiv p_1(z){\rm e}^{\alpha_1z^{k}}$, and there is an integer
$t\in\{1,2,\ldots,d\}$ such that $\alpha_{1}/\alpha_{2}=t/(n+1)$.
Consequently, $\gamma_2^n\left(\gamma'_2+\frac{\alpha_2}{n+1}kz^{k-1}\gamma_2\right)=p_2$.\par

{\bf Case 2:}\ \ Suppose that $W_1F''+W_2F'+W_3F\not\equiv 0$. Then \eqref{eq3.6} can be
rewritten as
\begin{equation}\label{eq3.14}
f^{n-1}(z)\psi(z,f)=-(W_1P''+W_2P'+W_3P),
\end{equation}
where
\begin{equation}\label{eq3.15}
\psi(z,f)=W_1\Big[n(n-1)\frac{f'}{f}(f')^2+3nf'f''+ff'''\Big]+W_2[n(f')^2+ff'']+W_3ff'.
\end{equation}
By \eqref{eq3.14}, \eqref{eq3.15} and Lemma \ref{le2.5}, we obtain
\begin{equation}\label{eq3.16}
m(r,\psi(z,f))=S(r,f).
\end{equation}
Combining \eqref{eq1.6} with \eqref{eq3.15} and \eqref{eq3.16}, we obtain
\begin{align*}
  2T(r,f)=T(r,f^2) &\leq m(r,\psi)+m\Big(r,\frac{f^2}{\psi}\Big)+S(r,f) \\
  &\leq N\Big(r,\frac{\psi}{f^2}\Big)+S(r,f) \\
  &\leq 3\overline{N}\Big(r,\frac{1}{f}\Big)+S(r,f),
\end{align*}
which implies that
\begin{equation*}
T(r,f)\leq\frac{3}{2}\overline{N}\left(r,\frac{1}{f}\right)+S(r,f).
\end{equation*}
\end{proof}

\section{ Proof of Theorem \ref{th1.2}}
\begin{proof}
Suppose that $f$ is a meromorphic solution of \eqref{eq1.6} of finite order.
The same argument as in the proof of Theorem \ref{th1.1} yields \eqref{eq3.1}.
Set $P:=P_d(z,f)$. Differentiating \eqref{eq1.6} results in
\begin{equation}\label{eq4.1}
nf^{n-1}(f')^{2}+f^{n}f''+P'=(p'_1+\alpha_{1}kz^{k-1}p_1){\rm e}^{\alpha_{1}z^{k}}
+(p'_2+\alpha_{2}kz^{k-1}p_2){\rm e}^{\alpha_{2}z^{k}}.
\end{equation}
Eliminating ${\rm e}^{\alpha_1z^{k}}$ or ${\rm e}^{\alpha_2z^{k}}$ from \eqref{eq1.6} and \eqref{eq4.1},
we obtain
\begin{equation}\label{eq4.2}
p_1[nf^{n-1}(f')^{2}+f^{n}f''+P']-(p'_1+\alpha_{1}kz^{k-1}p_1)(f^{n}f'+P)=A{\rm e}^{\alpha_{2}z^{k}},
\end{equation}
\begin{equation}\label{eq4.3}
p_2[nf^{n-1}(f')^{2}+f^{n}f''+P']-(p'_2+\alpha_{2}kz^{k-1}p_2)(f^{n}f'+P)=-A{\rm e}^{\alpha_{1}z^{k}},
\end{equation}
where $A=p_1p'_2-p'_1p_2+(\alpha_2-\alpha_1)kz^{k-1}p_1p_2$ and $A\not\equiv 0$.
Otherwise, $A\equiv 0$, then $p_2/p_1=c_1{\rm e}^{(\alpha_1-\alpha_2)z^{k}}$,
$c_1\in \mathbb{C}\setminus\{0\}$. Note that $T(r,p_j)=S(r,{\rm e}^{z^{k}})~(j=1,2)$,
then $\alpha_1=\alpha_2$, a contradiction.
Let $\beta_j=\sum\limits_{|\mu|=j}\Big(a_{\mu}(z)\prod\limits_{t=1}^s
\Big(\frac{f(z+d_t)}{f(z)}\Big)^{\mu_t}\Big)$,
it follows from Lemma \ref{le2.6} that
\begin{equation}\label{eq4.4}
\sigma(\beta_j)\leq k, \quad m(r,\beta_j)=S(r,f), \quad P_d(z,f)=\sum\limits_{j=0}^d\beta_jf^j(z).
\end{equation}
By \eqref{eq1.6} and \eqref{eq4.4}, we get
\begin{equation}\label{eq4.5}
\frac{1}{p_1{\rm e}^{\alpha_{1}z^{k}}+p_2{\rm e}^{\alpha_{2}z^{k}}-P_d(z,0)}\Bigg(\frac{f'}{f}
+\sum\limits_{j=1}^d\frac{\beta_{j}}{f^{n+1-j}}\Bigg)
=\Big(\frac{1}{f}\Big)^{n+1}.
\end{equation}
Since $P_d(z,0)\not\equiv 0$, it follows from Lemma \ref{le2.7} that
\begin{equation*}
\begin{split}
& m\Big(r,\frac{1}{p_1{\rm e}^{\alpha_{1}z^{k}}+p_2{\rm e}^{\alpha_{2}z^{k}}-P_d(z,0)}\Big)
=S(r,{\rm e}^{z^{k}}),\\
& m\Big(r,\frac{{\rm e}^{\alpha_{i}z^{k}}}{p_1{\rm e}^{\alpha_{1}z^{k}}+p_2{\rm e}^{\alpha_{2}z^{k}}-P_d(z,0)}\Big)
=S(r,{\rm e}^{z^{k}}),\quad i=1,2.
\end{split}
\end{equation*}
Combining \eqref{eq4.4}, \eqref{eq4.5}, the above two relations,
and the lemma on logarithmic derivatives, we have
\begin{equation}\label{eq4.6}
m\Big(r,\frac{1}{f}\Big)=S(r,{\rm e}^{z^{k}}), \quad
m\Big(r,\frac{{\rm e}^{\alpha_{i}z^{k}}}{f^{n+1}}\Big)=S(r,{\rm e}^{z^{k}}), \quad i=1,2.
\end{equation}
If $\frac{\alpha_{2}}{\alpha_{1}}=\frac{t}{n+1}$, $t\in\{1,2,\ldots,n-1\}$,
and the fact that
\begin{align*}
\Bigg|
\frac{{\rm e}^{\alpha_2z^{k}}}{f^{n-1}}\Bigg|=\Bigg|\frac{{\rm e}^{\frac{t}{n+1}\alpha_1z^{k}}
}{f^{n-1}}\Bigg|=\Bigg|\frac{{\rm e}^{\alpha_1z^{k}}}{f^{n+1}}\Bigg|^{\frac{t}{n+1}}\cdot\frac{1}{|f|^{n-1-t}},
\end{align*}
then it follows from \eqref{eq4.6} that
\begin{equation}\label{eq4.7}
m\Bigg(r,\frac{{\rm e}^{\alpha_2z^{k}}}{f^{n-1}}\Bigg)=S(r,{\rm e}^{z^{k}}).
\end{equation}
Applying \eqref{eq3.1}, \eqref{eq4.2}, \eqref{eq4.6}, \eqref{eq4.7} and Lemma \ref{le2.5},
Remark \ref{re2.3}, we conclude that
\begin{equation}\label{eq4.8}
m(r,\psi)=S(r,f),
\end{equation}
where
\begin{equation}\label{eq4.9}
\psi=np_1(f')^2+p_1ff''-(p'_1+\alpha_1kz^{k-1}p_1)ff'.
\end{equation}
Now we consider the following two cases:\par
{\bf Case 1:}\ \ $\psi \equiv 0$. Then $f^n(z)f'(z)=c_2p_1(z){\rm e}^{\alpha_1z^{k}}$,
$c_2\in \mathbb{C}\setminus\{0\}$, since $n\geq 2$, which implies that
any zero or pole of $f$ must also be a zero or pole of $p_1$.
Consequently, $N(r,f)+N(r,1/f)\leq N(r,p_1)+N(r,1/p_1)=S(r,f)$.
Since
\begin{equation*}
N\Big(r,\frac{c_2p_1(z){\rm e}^{\alpha_1z^{k}}}{f^{n+1}}\Big)
=N\Big(r,\frac{f^{n}f'}{f^{n+1}}\Big)=\overline{N}(r,f)+\overline{N}\Big(r,\frac{1}{f}\Big)=S(r,f),
\end{equation*}
\begin{equation*}
m\Big(r,\frac{c_2p_1(z){\rm e}^{\alpha_1z^{k}}}{f^{n+1}}\Big)
=m\Big(r,\frac{f^{n}f'}{f^{n+1}}\Big)=S(r,f),
\end{equation*}
which implies that $T\Big(r,\frac{c_2p_1{\rm e}^{\alpha_1z^{k}}}{f^{n+1}}\Big)=S(r,f)$.
Set $h(z)=\frac{f^{n+1}(z)}{\eta_1p_1(z){\rm e}^{\alpha_1z^{k}}}$, then $T(r,h)=S(r,f)$,
whence $f^{n+1}(z)=c_2p_1(z)h(z){\rm e}^{\alpha_1z^{k}}$,
$f(z)=\gamma_1(z){\rm e}^{\frac{\alpha_1}{n+1}z^{k}}$. Substituting these expressions into
\eqref{eq1.6}, we have
\begin{equation*}
c_2p_1(z){\rm e}^{\alpha_1z^{k}}+\beta_0(z)+
\sum\limits_{j=1}^d\beta_j(z)\gamma_1^{j}(z){\rm e}^{\frac{j\alpha_1}{n+1}z^{k}}
=p_1(z){\rm e}^{\alpha_1z^{k}}+p_2(z){\rm e}^{\alpha_2z^{k}}.
\end{equation*}
Combining the above equation and Lemma \ref{le2.4}, we get
$\beta_0(z)=P_d(z,0)\equiv 0$, a contradiction.\par

{\bf Case 2:}\ \ $\psi\not\equiv 0$. If $z_0$ is a multiple zero of $f$, then $z_0$ must be a
zero of $\psi$, by \eqref{eq4.8} and \eqref{eq4.9}, we have
$N_{(2}(r,\frac{1}{f})\leq N(r,\frac{1}{\psi})\leq T(r,\frac{1}{\psi})=S(r,f)$.
It follows from \eqref{eq3.1} and \eqref{eq4.6} that
\begin{align*}
T(r,f)=N_{1)}\Big(r,\frac{1}{f}\Big)+S(r,f).
\end{align*}
Differentiating \eqref{eq4.9} results in
\begin{align}\label{eq4.10}
\begin{array}{l}
\psi'=[(n-1)p'_1-\alpha_1kz^{k-1}p_1](f')^2+(2n+1)p_1f'f''-(p'_1+\alpha_1kz^{k-1}p_1)'ff'\\
-\alpha_1kz^{k-1}p_1ff''+p_1ff'''.
\end{array}
\end{align}
Combining \eqref{eq4.9} and \eqref{eq4.10}, we obtain
\begin{align}\label{eq4.11}
\begin{split}
& [((n-1)p'_1-\alpha_1kz^{k-1}p_1)\psi-np_1\psi'](f')^2+(2n+1)p_1\psi f'f''+\\
& [(p'_1+\alpha_1kz^{k-1}p_1)\psi'-(p'_1+\alpha_1kz^{k-1}p_1)'\psi]ff'
-(p_1\psi'+\alpha_1kz^{k-1}p_1\psi)ff''+p_1\psi ff'''=0.
\end{split}
\end{align}
Suppose that $z_1$ is a simple zero of $f$ which is not the zero of coefficients of \eqref{eq4.11}.
By \eqref{eq4.11}, we see that $z_1$ is a zero of
$[((n-1)p'_1-\alpha_1kz^{k-1}p_1)\psi-np_1\psi']f'+(2n+1)p_1\psi f''$.
Set
\begin{equation}\label{eq4.12}
\beta=\frac{[((n-1)p'_1-\alpha_1kz^{k-1}p_1)\psi-np_1\psi']f'+(2n+1)p_1\psi f''}{f}.
\end{equation}
Then $T(r,\beta)=S(r,f)$. It follows from \eqref{eq4.12} that
\begin{equation}\label{eq4.13}
f''=\frac{1}{2n+1}\left[n\frac{\psi'}{\psi}+\alpha_1kz^{k-1}-(n-1)\frac{p'_1}{p_1}\right]f'
+\frac{\beta}{(2n+1)p_1\psi}f.
\end{equation}
Substituting \eqref{eq4.13} into \eqref{eq4.9}, we have
\begin{equation}\label{eq4.14}
af^2+bff'+c(f')^2=\psi,
\end{equation}
where
\begin{equation}\label{eq4.15}
\left\{\begin{array}{l}
a=\frac{\beta}{(2n+1)\psi},\\
b=\frac{np_1}{2n+1}\left(\frac{\psi'}{\psi}-2\alpha_1kz^{k-1}-3\frac{p'_1}{p_1}\right),\\
c=np_1.
\end{array}\right.
\end{equation}
From \eqref{eq4.15} and Lemma \ref{le2.8}, we conclude that
\begin{equation}\label{eq4.16}
  c(b^2-4ac)\frac{\psi'}{\psi}=c(b^2-4ac)'-b(b^2-4ac)-c'(b^2-4ac).
\end{equation}
Next we consider two subcases as follows.\par
{\bf Subcase 2.1:}\ \ If $b^2-4ac\not\equiv 0$. Combining \eqref{eq4.15} and \eqref{eq4.16},
we get
\begin{equation}\label{eq4.17}
  2\alpha_1kz^{k-1}=2(n+1)\frac{\psi'}{\psi}-(2n+1)\frac{(b^2-4ac)'}{b^2-4ac}-2(n-1)\frac{p'_1}{p_1}.
\end{equation}
Integrating both sides of \eqref{eq4.17} gives
\begin{align*}
{\rm e}^{2\alpha_1 z^{k}}=c_3\psi^{2(n+1)}\cdot (b^2-4ac)^{-(2n+1)}\cdot p_1^{-2(n-1)},\quad
c_3\in \mathbb{C}\setminus\{0\},
\end{align*}
which implies that $T(r,{\rm e}^{2\alpha_1 z^{k}})=S(r,f)$, a contradiction.\par

{\bf Subcase 2.2:}\ \ If $b^2-4ac\equiv 0$. Differentiating \eqref{eq4.14} gives
\begin{equation}\label{eq4.18}
\psi'=a'f^2+(2a+b')ff'+(b+c')(f')^2+bff''+2cf'f''.
\end{equation}
Suppose that $z_{2}$ is a simple zero of $f$ which is not the zero of $a,b,c$.
By \eqref{eq4.14} and \eqref{eq4.18}, we get $z_{2}$ is a zero of $[(b+c')\psi-c\psi']f'+2c\psi f''$.
Set
\begin{equation}\label{eq4.20}
\gamma=\frac{[(b+c')\psi-c\psi']f'+2c\psi f''}{f}.
\end{equation}
Then $T(r,\gamma)=S(r,f)$, and
\begin{equation}\label{eq4.20}
f''=\frac{1}{2}\left(\frac{\psi'}{\psi}-\frac{b+c'}{c}\right)f'+\frac{\gamma}{2c\psi}f.
\end{equation}
Substituting \eqref{eq4.20} into \eqref{eq4.18} yields
\begin{equation}\label{eq4.21}
\psi'=\tau_1f^2+\tau_2ff'+\tau_3(f')^2,
\end{equation}
where
\begin{align*}
\tau_1=a'+\frac{b\gamma}{2c\psi},
~~\tau_2=2a+b'+\frac{b}{2}\left(\frac{\psi'}{\psi}-\frac{b+c'}{c}\right)+\frac{\gamma}{\psi},
~~\tau_3=c\frac{\psi'}{\psi}.
\end{align*}
Eliminating $(f')^2$ from \eqref{eq4.14} and \eqref{eq4.21}, we get
\begin{equation}\label{eq4.22}
A_1f+A_2f'=0,
\end{equation}
where
\begin{align*}
A_1 &=\tau_1-a\frac{\psi'}{\psi}=a'+\frac{b\gamma}{2c\psi}-a\frac{\psi'}{\psi},\\
A_2 & =\tau_2-b\frac{\psi'}{\psi}=
2a+b'-\frac{b}{2}\left(\frac{\psi'}{\psi}+\frac{b+c'}{c}\right)+\frac{\gamma}{\psi},
\end{align*}
and $T(r,A_{j})=S(r,f)~j=1,2$. Suppose that $z_{3}$ is a simple zero of $f$
which is not the zero of $A_1,A_2$. By \eqref{eq4.22}, we deduce that $A_1=A_2\equiv 0$.
It follows from \eqref{eq4.20} that
\begin{equation}\label{eq4.23}
f''=\frac{1}{2}\left(\frac{\psi'}{\psi}-\frac{b+c'}{c}\right)f'-\frac{1}{b}\left(a'-a\frac{\psi'}{\psi}\right)f,
\end{equation}
where $b=\frac{np_1}{2n+1}\left(\frac{\psi'}{\psi}-2\alpha_1kz^{k-1}-3\frac{p'_1}{p_1}\right)\not\equiv 0$.
Otherwise, $b\equiv 0$, which implies that
${\rm e}^{2\alpha_1z^{k}}=c_4\psi\cdot p_1^{-3}, c_4\in \mathbb{C}\setminus\{0\}$,
then $T(r,{\rm e}^{2\alpha_1z^{k}})=S(r,f)$, a contradiction.
Substituting $b^2=4ac$ and \eqref{eq4.15} into \eqref{eq4.23} results in
\begin{align}\label{eq4.24}
\begin{array}{l}
f''=\frac{1}{2n+1}\Big[n\frac{\psi'}{\psi}+\alpha_1kz^{k-1}-(n-1)\frac{p'_1}{p_1}\Big]f'-\\
\frac{1}{4(2n+1)}\Big[\Big(\frac{\psi'}{\psi}-2\alpha_1kz^{k-1}-3\frac{p'_1}{p_1}\Big)
\Big(\frac{p'_1}{p_1}-\frac{\psi'}{\psi}\Big)
+2\Big(\frac{\psi'}{\psi}-2\alpha_1kz^{k-1}-3\frac{p'_1}{p_1}\Big)'\Big]f.
\end{array}
\end{align}
Combining \eqref{eq4.13} and \eqref{eq4.24}, we obtain
\begin{equation}\label{eq4.25}
\frac{\beta}{p_1\psi}=-\frac{1}{4}
\Big[\Big(\frac{\psi'}{\psi}-2\alpha_1kz^{k-1}-3\frac{p'_1}{p_1}\Big)
\Big(\frac{p'_1}{p_1}-\frac{\psi'}{\psi}\Big)
+2\Big(\frac{\psi'}{\psi}-2\alpha_1kz^{k-1}-3\frac{p'_1}{p_1}\Big)'\Big].
\end{equation}
Set $u=\frac{\psi'}{\psi}-2\alpha_1kz^{k-1}-3\frac{p'_1}{p_1}(\not\equiv 0)$, $v=\frac{p'_1}{p_1}-\frac{\psi'}{\psi}$.
Then \eqref{eq4.25} takes the form
\begin{equation}\label{eq4.26}
\frac{\beta}{\psi}=-\frac{p_1}{4}(uv+2u').
\end{equation}
By \eqref{eq4.15}, $b^2=4ac$ and $u=\frac{\psi'}{\psi}-2\alpha_1kz^{k-1}-3\frac{p'_1}{p_1}(\not\equiv 0)$, we get
\begin{equation}\label{eq4.27}
\frac{\beta}{\psi}=\frac{p_1}{4}\cdot\frac{nu^{2}}{2n+1}.
\end{equation}
It follows from \eqref{eq4.26}, \eqref{eq4.27}, and the expressions of $u$ and $v$ that
\begin{align}\label{eq4.28}
2n\alpha_1kz^{k-1}=2(2n+1)\frac{u'}{u}-(n+1)\frac{\psi'}{\psi}-(n-1)\frac{p'_1}{p_1}.
\end{align}
Integrating both sides of \eqref{eq4.28} gives
\begin{align*}
{\rm e}^{2n\alpha_1z^{k}}=c_5u^{2(2n+1)}\cdot \psi^{-(n+1)}\cdot p_1^{-(n-1)},\quad
c_5\in \mathbb{C}\setminus\{0\},
\end{align*}
consequently, $T(r,{\rm e}^{2n\alpha_1z^{k}})=S(r,f)$, a contradiction.\par

Similarly, if $\frac{\alpha_{2}}{\alpha_{1}}=\frac{n+1}{t}$ for some integer $t\in\{1,2,\ldots,n-1\}$,
then Eq.\eqref{eq1.6} admits no meromorphic solution of finite order. Thereby completing the proof.\par
\end{proof}

\section{ Proof of Theorem \ref{th1.3}}
\begin{proof}
Suppose that $f$ is a meromorphic solution of \eqref{eq1.9} of finite order.
We first show that $f$ is entire. If not, $f$ has a pole $z_1$ of
multiplicity $k_1(\geq 1)$. By \eqref{eq1.9} and the hypotheses of Theorem \ref{th1.3},
we may choose $\nu\in\{1,2,\ldots,s\}$ such that $z_2:=z_1+d_{\nu}$ is a pole of
$f(z)$ of multiplicity $k_2>\frac{n+1}{n-1}k_1$. Continuing inductively,
we obtain a sequence $\{z_j\}$ of poles of $f(z)$ whose multiplicities satisfy
$k_j>\frac{n+1}{n-1}k_{j-1}~(j=2,3,\ldots)$. Since $\frac{n+1}{n-1}>1$,
this gives $\lambda\left(\frac{1}{f}\right)=+\infty$.
This contradicts $\sigma(f)<+\infty$. Hence $f$ is entire.
The function $p_{1}{\rm e}^{\lambda z}+p_{2}{\rm e}^{-\lambda z}$
is transcendental, so every meromorphic solution $f$ of \eqref{eq1.9} of finite order is
transcendental as well. Moreover $N(r,p_{1}{\rm e}^{\lambda z}+p_{2}{\rm e}^{-\lambda z})=0$ and
$\sigma(p_{1}{\rm e}^{\lambda z}+p_{2}{\rm e}^{-\lambda z})=1$, and it follows from Lemma \ref{le2.1}, we have
\begin{equation}\label{eq5.1}
\sigma(f)=1.
\end{equation}
Setting $P:=P_d(z,f)$ and differentiating \eqref{eq1.9}, we obtain
\begin{align}\label{eq5.2}
nf^{n-1}(f')^{2}+f^{n}f''+P'=\lambda p_{1}{\rm e}^{\lambda z}-\lambda p_{2}{\rm e}^{-\lambda z}.
\end{align}
Eliminating ${\rm e}^{\lambda z}$ and ${\rm e}^{-\lambda z}$ from \eqref{eq1.9} and \eqref{eq5.2},
respectively, we obtain
\begin{align}\label{eq5.3}
f^{n-1}\chi_1+Q_1=-2\lambda p_{2}{\rm e}^{-\lambda z},
\end{align}
\begin{align}\label{eq5.4}
f^{n-1}\chi_2+Q_2=2\lambda p_{1}{\rm e}^{\lambda z},
\end{align}
where
$\chi_1=n(f')^{2}+ff''-\lambda ff'$, $\chi_2=n(f')^{2}+ff''+\lambda ff'$, $Q_1=P'-\lambda P$,
$Q_2=P'+\lambda P$. Differentiating \eqref{eq5.3} and eliminating ${\rm e}^{-\lambda z}$ yields
\begin{align}\label{eq5.5}
f^{n-1}\varphi=R,
\end{align}
where
\begin{align}\label{eq5.6}
\varphi=\lambda f\chi_1+(n-1)f'\chi_1+f\chi'_1,
\end{align}
\begin{align}\label{eq5.7}
R=-\lambda Q_1-Q'_1.
\end{align}
Since $\deg_{f}R\leq n-2$, it follows from \eqref{eq5.5}, Lemma \ref{le2.5} and Remark \ref{re2.3},
that
\begin{align}\label{eq5.8}
m(r,\varphi)=O(\log r).
\end{align}
As $f$ is an entire, and hence \eqref{eq5.8} gives
\begin{align}\label{eq5.9}
T(r,\varphi)=m(r,\varphi)=O(\log r).
\end{align}
We consider two cases:\\
{\bf Case 1:}\ \ Suppose that $\varphi\equiv 0$. Then $\lambda f\chi_1+(n-1)f'\chi_1+f\chi'_1\equiv 0$.
If $\chi_1\equiv 0$, then $n(f')^{2}+ff''-\lambda ff'\equiv 0$, that is
$f^{n}f'=c_1{\rm e}^{\lambda z}$ for some $c_1\in \mathbb{C}\setminus\{0\}$, and hence
$f(z)=\gamma_1{\rm e}^{\frac{\lambda}{n+1}z}$. Substituting this into \eqref{eq1.9}
and applying Lemma \ref{le2.4}, we obtain $p_2=0$, a contradiction. Therefore $\chi_1\not\equiv 0$.
It then follows from $\varphi\equiv 0$ that $f^{n-1}\chi_1=c_2{\rm e}^{-\lambda z}$ for some
$c_2\in \mathbb{C}\setminus\{0\}$. Substituting this into \eqref{eq5.3} yields
$\left(1+\frac{2\lambda p_2}{c_2}\right)f^{n-1}\chi_1=-Q_1$. If $1+\frac{2\lambda p_2}{c_2}\neq 0$,
then, since $n\geq 3$, Lemma \ref{le2.5} and Remark \ref{re2.3} give $m(r,\chi_1)=O(\log r)$,
$m(r,f\chi_1)=O(\log r)$. As $f$ is an entire, then $T(r,\chi_1)=O(\log r)$, $T(r,f\chi_1)=O(\log r)$,
this leading to $T(r,f)\leq T(r,f\chi_1)+T(r,\frac{1}{\chi_1})=O(\log r)$, a contradiction.
If $1+\frac{2\lambda p_2}{c_2}=0$, then $Q_1=0$ and consequently $P=c_3{\rm e}^{\lambda z}$ for some
$c_3\in \mathbb{C}\setminus\{0\}$. By \eqref{eq5.4},
$f^{n-1}\chi_2+Q_2=\frac{2\lambda p_1}{c_3}P$, so Lemma \ref{le2.5} and Remark \ref{re2.3}
yield $m(r,\chi_2)=O(\log r)$, $m(r,f\chi_2)=O(\log r)$. If $\chi_2\equiv 0$,arguing as in the case
$\chi_1\equiv 0$, we obtain $p_1=0$, a contradiction.
If $\chi_2\not\equiv 0$, arguing as in the case $\chi_1\not\equiv 0$,
we obtain $T(r,f)\leq O(\log r)$, again a contradiction.\par

{\bf Case 2:}\ \
If $\varphi\not\equiv 0$. Substituting $\chi_1=n(f')^{2}+ff''-\lambda ff'$ and
$\chi'_1=-\lambda(f')^{2}+(2n+1)f'f''-\lambda ff''+ff'''$ into \eqref{eq5.6} yield
\begin{align}\label{eq5.10}
\varphi=n(n-1)(f')^{3}+3nff'f''-\lambda^{2}f^{2}f'+f^{2}f'''.
\end{align}
Recall that $\varphi\not\equiv 0$. Then\eqref{eq5.9} and \eqref{eq5.10} give
$m(r,\frac{1}{f})=O(\log r)$. Moreover, by \eqref{eq5.10}, every multiple zero of $f$ is
also a zero of $\varphi$, whence $N_{(2}(r,\frac{1}{f})\leq T(r,\varphi)=O(\log r)$.
Consequently,
\begin{align}\label{eq5.11}
T(r,f)=N_{1)}\Big(r,\frac{1}{f}\Big)+O(\log r).
\end{align}
Differentiating \eqref{eq5.10} results in
\begin{align}\label{eq5.12}
\varphi'=3n^{2}(f')^{2}f''+3nf(f'')^{2}-2\lambda^{2}f(f')^{2}+(3n+2)ff'f'''
-\lambda^{2}f^{2}f''+f^{2}f^{(4)}.
\end{align}
Suppose that $z_1$ is a simple zero of $f$. By \eqref{eq5.10} and \eqref{eq5.12},
we see that $z_1$ is a zero of $(n-1)\varphi'f'-3n\varphi f''$.
Set
\begin{align}\label{eq5.13}
\zeta(z)=\frac{(n-1)\varphi'f'-3n\varphi f''}{f},
\end{align}
then $T(r,\zeta)=O(\log r)$, that is $\zeta(z)$ is a rational function.
It follows from \eqref{eq5.13} that
\begin{align}\label{eq5.14}
f''=A_1f'+B_1f,
\end{align}
where $A_1=\frac{(n-1)\varphi'}{3n\varphi}$, $B_1=-\frac{\zeta}{3n\varphi}$ are rational functions.
It follows from \eqref{eq5.14} that
\begin{align}\label{eq5.15}
f'''=(A_1^{2}+A'_1+B_1)f'+(A_1B_1+B'_1)f,
\end{align}
\begin{align}\label{eq5.16}
f^{(4)}=(A_1^{3}+3A_1A'_1+2A_1B_1+A''_1+2B'_1)f'+(A_1^{2}B_1+2A'_1B_1+A_1B'_1+B_1^{2}+B''_1)f.
\end{align}
Substituting \eqref{eq5.14} and \eqref{eq5.15} into \eqref{eq5.10} yields
\begin{align}\label{eq5.17}
\varphi=a_1(f')^{3}+a_2f(f')^{2}+a_3f^{2}f'+a_4f^{3},
\end{align}
where
\begin{align*}
\left\{\begin{array}{l}
a_1=n(n-1),\\
a_2=3nA_1,\\
a_3=A_1^{2}+A'_1+(3n+1)B_1-\lambda^{2},\\
a_4=A_1B_1+B'_1,
\end{array}\right.
\end{align*}
are rational functions.
Substituting \eqref{eq5.14}-\eqref{eq5.16} into \eqref{eq5.12} yields
\begin{align}\label{eq5.18}
\varphi'=b_1(f')^{3}+b_2f(f')^{2}+b_3f^{2}f'+b_4f^{3},
\end{align}
where
\begin{align*}
\left\{\begin{array}{l}
b_1=3n^{2}A_1,\\
b_2=2(3n+1)A_1^{2}+(3n+2)A'_1+(3n^{2}+3n+2)B_1-2\lambda^{2},\\
b_3=A_1^{3}+3A_1A'_1+(9n+4)A_1B_1+A''_1-\lambda^{2}A_1+(3n+4)B'_1,\\
b_4=A_1^{2}B_1+2A'_1B_1+A_1B'_1+(3n+1)B_1^{2}-\lambda^{2}B_1+B''_1,
\end{array}\right.
\end{align*}
are rational functions. Substituting $A_1=\frac{n-1}{3n}\frac{\varphi'}{\varphi}$,
we obtain $b_1-a_1\frac{\varphi'}{\varphi}=3n^{2}A_1-n(n-1)\frac{\varphi'}{\varphi}=0$.
Combining \eqref{eq5.17} and \eqref{eq5.18}, we obtain
\begin{align}\label{eq5.19}
\Big(b_2-a_2\frac{\varphi'}{\varphi}\Big)f(f')^{2}+
\Big(b_3-a_3\frac{\varphi'}{\varphi}\Big)f^{2}f'+
\Big(b_4-a_4\frac{\varphi'}{\varphi}\Big)f^{3}=0.
\end{align}
Since $f$ is entire, \eqref{eq5.11} and \eqref{eq5.19} imply that
\begin{align*}
b_2-a_2\frac{\varphi'}{\varphi} &
=2(3n+1)A_1^{2}+(3n+2)A'_1+(3n^{2}+3n+2)B_1-2\lambda^{2}-3nA_1\frac{\varphi'}{\varphi}\equiv 0.\\
b_3-a_3\frac{\varphi'}{\varphi} &
=A_1^{3}+3A_1A'_1+(9n+4)A_1B_1+A''_1-\lambda^{2}A_1+(3n+4)B'_1\\
&-[A_1^{2}+A'_1+(3n+1)B_1-\lambda^{2}]\frac{\varphi'}{\varphi}\equiv 0.\\
b_4-a_4\frac{\varphi'}{\varphi} &
=A_1^{2}B_1+2A'_1B_1+A_1B'_1+(3n+1)B_1^{2}-\lambda^{2}B_1+B''_1-(A_1B_1+B'_1)\frac{\varphi'}{\varphi}\equiv 0.
\end{align*}
From $b_2-a_2\frac{\varphi'}{\varphi}\equiv 0$, we obtain
\begin{align*}
2(3n+1)A_1^{2}+(3n+2)A'_1+(3n^{2}+3n+2)B_1-2\lambda^{2}=3nA_1\frac{\varphi'}{\varphi}.
\end{align*}
As $z\to \infty$, $\frac{\varphi'}{\varphi}\to 0$, and hence
$A_1=\frac{(n-1)\varphi'}{3n\varphi}\to 0$. Consequently, $(3n^{2}+3n+2)B_1-2\lambda^{2}\to 0$.
Recalling that $B_1=-\frac{\zeta}{3n\varphi}$, we obtain
\begin{align}\label{eq5.20}
\lim_{z\to \infty}\frac{\zeta}{\varphi}=-\frac{6n\lambda^{2}}{3n^{2}+3n+2}.
\end{align}
Dividing $b_4-a_4\frac{\varphi'}{\varphi}\equiv 0$ by $B_1$ yields
\begin{align*}
A_1^{2}+2A'_1+A_1\frac{B'_1}{B_1}+\frac{B''_1}{B_1}+(3n+1)B_1-\lambda^{2}
=\Big(A_1+\frac{B'_1}{B_1}\Big)\frac{\varphi'}{\varphi}.
\end{align*}
As $z\to \infty$, $\frac{\varphi'}{\varphi}\to 0$, and hence $A_1=\frac{(n-1)\varphi'}{3n\varphi}\to 0$;
moreover, $\frac{B_1'}{B_1}\to 0$, $\frac{B''_1}{B_1}\to 0$. Consequently, $(3n+1)B_1-\lambda^{2}\to 0$.
Recalling that $B_1=-\frac{\zeta}{3n\varphi}$, we obtain
\begin{align}\label{eq5.21}
\lim_{z\to \infty}\frac{\zeta}{\varphi}=-\frac{3n\lambda^{2}}{3n+1}.
\end{align}
Comparing \eqref{eq5.20} and \eqref{eq5.21} gives
$\frac{6n}{3n^{2}+3n+2}=\frac{3n}{3n+1}$,
so that $n(n-1)=0$, i.e., $n=0$ or $n=1$, contradicting $n\ge 3$.\par

\end{proof}


\scriptsize
\begin{center}

\end{center}

\noindent Min Feng Chen\\
School of Mathematics and Statistics\\
Guangdong University of Foreign Studies\\
Guangzhou, 510006\\
P. R. China\\
Email: minfengchen2018@gdufs.edu.cn\\

\noindent Zong-Sheng Gao\\
LMIB and School of Mathematical Sciences\\
Beihang University\\
Beijing, 100191\\
P. R. China\\
E-mail: zshgao@buaa.edu.cn\\

\end{document}